\documentclass{amsart}
\usepackage{graphicx}
\usepackage{amssymb}
\usepackage{longtable}
\usepackage[pdfauthor={Richard J. Mathar},pdfkeywords={Diophantine Equation, Pell Equation},pdfsubject={Mathematics, Number Theory},colorlinks=true,citecolor=blue]{hyperref}

\usepackage{amsmath}

\newtheorem{thm}{Theorem}
\newtheorem{lma}{Lemma}
\newtheorem{exa}{Example}
\newtheorem{defn}{Definition}
\newtheorem{remark}{Remark}
\newtheorem{alg}{Algorithm}
\begin{document}

\title[Constructing Bhaskara Pairs]{Construction of Bhaskara Pairs}

\author{Richard J. Mathar}
\urladdr{http://www.mpia.de/~mathar}
\address{Max-Planck Institute of Astronomy, K\"onigstuhl 17, 69117 Heidelberg, Germany}


\date{\today}

\begin{abstract}
We construct integer solutions $\{a,b\}$ to the coupled system of diophantine 
quadratic-cubic equations 
$a^2+b^2=x^3$ and $a^3+b^3=y^2$ for fixed ratios $a/b$.
\end{abstract}

\maketitle 

\section{Pair of Coupled Nonlinear Diophantine Equations} 
\subsection{Scope}

Following a nomenclature of Gupta we define \cite[\S 4.4]{Cooke}:
\begin{defn}(Bhaskara pair) \label{def.pair}
A Bhaskara pair is a pair $\{a,b\}$ of integers that solve the system of 
two nonlinear Diophantine equations of Fermat type:
\begin{equation}
a^2+b^2=x^3 \wedge
a^3+b^3=y^2 \\
\label{eq.defn}
\end{equation}
for some pair $\{x,y\}$.
\end{defn}
\begin{remark}
Lists of $a$ and $b$ are gathered in the Online Encyclopedia
of Integer Sequences \cite[A106319,A106320]{EIS}.
\end{remark}

The symmetry swapping $a$ and $b$ in the equations 
indicates that without loss of information we can
assume $0\le a\le b$, denoting the larger member of the pair
by $b$.

We will not look into solutions where $a$ or $b$ are rational integers (\emph{fractional} Bhaskara pairs).

The two equations can be solved individually \cite{BennettAA163,Dahmenarxiv1002,BruinCM118}.

\begin{alg}\label{alg.nonf}
Given any solution $\{a,b\}$, further solutions $\{as^6,bs^6\}$ are derived
by multiplying both $a$ and $b$
by a sixth power of a common integer $s$, multiplying at the same time
on the right hand sides
$x$ by $s^4$ and $y$ by $s^9$. 
\end{alg}
\begin{defn}
(Fundamental Bhaskara Pair)
A \emph{fundamental} Bhaskara pair is a Bhaskara pair
$\{a,b\}$ where $a$ and $b$ have no common
divisor which is 6-full---meaning there is no prime $p$ such that
$p^6\mid a$ and $p^6\mid b$.
\end{defn}

Although fundamental solutions are pairs that do not have
a common divisor that is a non-trivial sixth power, \emph{individually}
$a$ or $b$ of a fundamental pair may contain sixth or higher (prime) powers.
\begin{exa}\label{ex.pair}
The following is a fundamental Bhaskara pair with 
$2^6\mid a$, $2^6\nmid b$:
$a=2^6\times 5^4\times 31^3\times 61^3$,
$b= 5^4\times 31^3\times 61^3\times 83$,
$x=5^3\times 13\times 31^2\times 61^2$, and
$y=3\times 5^6\times 7\times 31^5\times 61^5$.
\end{exa}

\section{Trivial Solutions}
\subsection{Primitive Solutions}
A first family of solutions is found by setting $a=0$. This reduces the equations to
\begin{equation}
b^2=x^3 \wedge b^3=y^2.
\end{equation}
$x^3$ must be a perfect cube, so in the canonical prime power factorization of $x^3$ all exponents of the primes
must be multiples of three. Also in the canonical prime power factorization of $b^2$ all exponents
must be even. So the first equation demands that the exponents on both sides must be multiples of $[2,3]=6$.
\begin{defn}
Square brackets $[.,.]$ denote the least common multiple. 
Parenthesis $(.,.)$ denote the greatest common divisor.
\end{defn}
In consequence all $b$ must be perfect cubes.
Likewise the second equation demands that the exponents
of $b^3$ and of $y^2$ are multiples of 6. 
In consequence all $b$ must be perfect squares.
Uniting both requirements, all $b$ must be perfect sixth powers. And this requirement
is obviously also sufficient: perfect sixth powers \cite[A001014]{EIS} generate Bhaskara pairs:
\begin{thm}\label{thm.a0}
All integer pairs $\{0,n^6\}$, $n\in {\mathbb Z}_0$, are Bhaskara pairs.
The associated right hand sides are
$x=n^4$, $y=n^9$.
\end{thm}

\subsection{Bhaskara Twins}
\begin{defn}(Bhaskara Twins)
Bhaskara twins are a Bhaskara pair where $a=b$.
\end{defn}
According to Definition \ref{def.pair} the Bhaskara twins 
\cite[A106318]{EIS} solve
\begin{equation}
2a^2=x^3 \wedge 2a^3=y^2.
\label{eq.k1}
\end{equation}
Working modulo 2 in the two equations requires that $x^3$ and $y^2$ are even, so $x$ and $y$ must be even,
say $x=2\alpha$, $y=2\beta$. So
\begin{equation}
a^2=4\alpha^3 \wedge a^3=2\beta ^2.
\label{eq.k1red}
\end{equation}
The first equation requires 
by the right hand side
that in the canonical prime power factorization of both sides the exponents
of the odd primes are multiples of 3 and that the exponent of the prime 2 is $\equiv 2 \pmod 3$.
By the left hand side of the first equation it requires that all exponents are even.
So the exponents of the odd primes are multiples of 6, and the exponent of 2 is $\equiv 2 \pmod 6$. 
So from the first equation $a=2^{1+3\times}3^{3\times}5^{3\times}\cdots$, which means $a$ is twice a third power.
\begin{defn}
The notation $3\times$ in the exponents means ``any multiple of 3.''
\end{defn}

The second equation in \eqref{eq.k1red} demands by the right hand side
that the exponents of the odd primes are even and that the exponent of 2 is $\equiv 1 \pmod 2$.
Furthermore by the left hand side all exponents are multiples of 3. This means
all exponents of the odd primes are multiples of 6, and the exponent of the prime 2 is $\equiv 3 \pmod 6$
So from the second equation $a=2^{1+2\times}3^{2\times}5^{2\times}\cdots$, which means $a$ must be twice a perfect square.
Uniting both requirements, $a$ must be twice a sixth power. Obviously that requirement
is also sufficient to generate solutions:
\begin{thm}\label{thm.twin}
The Bhaskara Twins are the integer pairs $\{2n^6,2n^6\}$, $n\in {\mathbb Z}_0$.
The associated free variables are $x=2n^4$, $y=4n^6$.
\end{thm}

\section{Rational Ratios of the two Members} \label{sec.k}
\subsection{Prime Factorization} \label{sec.kpr}
The general solution to \eqref{eq.defn} is characterized by some ratio $a/b = u/k \le 1$ with some coprime pair of integers $(k,u)=1$. Cases where $u$ and $k$ are not coprime
are not dealt with because they do not generate new solutions.

If $k$ were not a divisor of $b$, $a=ub/k$ would require
that $k$ is a divisor of $u$ to let $a$ be integer, contradicting the requirement that $u$ and $k$ are coprime.
\begin{alg}
We only admit the denominators $k\mid b$.
\end{alg}

Theorem \ref{thm.a0} and \ref{thm.twin}
cover the solutions of the special cases $u=0$ or $u=1$.
Introducing the notation into \eqref{eq.defn} yields
\begin{equation}
(1+u^2/k^2)b^2=x^3 \wedge
(1+u^3/k^3)b^3=y^2 ;
\end{equation}
\begin{equation}
(u^2+k^2)b^2=k^2x^3 \wedge
(u^3+k^3)b^3=k^3y^2 .
\label{eq.ksol}
\end{equation}

\begin{table}
\begin{tabular}{r|r|r|r}
$k$ & $1+k^2$ & $1+k^3$  & $k$ \\
\hline
1 & 2 & 2 & 1 \\
2 & 5 & $3^2$ &2 \\
3 & $2\times 5$  & $2^2\times 7$ & 3 \\
4 & 17 & $5\times 13$ & $2^2$ \\
5 & $2\times 13$ & $2\times 3^2\times 7$ & $5$ \\
6 & $37$ & $7\times 31$ & $2\times 3$ \\
\end{tabular}
\caption{Prime factorizations of $1+k^2$, $1+k^3$ and $k$}
\label{tab.pfac}
\end{table}

Define prime power exponents $c_i$, $d_i$, $b_i$, $x_i$ and $y_i$ as follows
by prime power factorizations,
where $p_i$ is the $i$-th prime:
\begin{eqnarray} \label{eq.ukdef}
u^2+k^2 
&=& \prod_i p_i^{c_i}, \label{eq.cdef} \\
u^3+k^3 
&=& \prod_i p_i^{d_i}, \label{eq.ddef} \\
b
&=& \prod_i p_i^{b_i},\\
k
&=& \prod_i p_i^{k_i},\\
x
&=& \prod_i p_i^{x_i},\\
y
&=& \prod_i p_i^{y_i}.
\end{eqnarray}
In \eqref{eq.cdef}, $u^2+k^2$ is the sum of two squares \cite[A000404]{EIS}. 
Because $u$ and $k$ are coprime, these $u^2+k^2$
are 2, 5, 10, 13, 17, 25, 26, 29, 34, 37, 41,\ldots,
numbers whose prime divisors are all $p\equiv 1 \pmod 4$
with the exception of a single factor of $2$ 
\cite[A008784]{EIS}\cite[Thm.\ 2.5]{Moreno}\cite[Thm.\ 3]{Grosswald}:
\begin{lma} \label{lma.hypot}
\begin{equation}
c_1 \in \{0,1\}.
\end{equation}
\begin{equation}
p_i \equiv 1 \pmod 4,\,\mathrm{if}\, c_i >0\wedge
p_i\ge 3.
\end{equation}
\end{lma}
\begin{exa}
If $u=2^6$, $k=83$ as in Example \ref{ex.pair},
$u^2+k^2=5\times 13^3$, so $c_3=1$, $c_6=3$,
and $u^3+k^3=3^2\times 7^2\times 31\times 61$,
so $d_2=2$, $d_4=2$, $d_{11}=1$, $d_{18}=1$.
\end{exa}

The uniqueness of the prime power representations in (\ref{eq.ksol}) requires for all $i\ge 1$
\begin{subequations}\label{eq.crtpre}
\begin{eqnarray}
c_i + 2b_i &=& 2k_i+3x_i ,
\label{eq.modi1}
\\
d_i + 3b_i &=& 3k_i+2y_i,
\label{eq.modi2}
\end{eqnarray}
\end{subequations}
for unknown sets of ${b_i,x_i,y_i}$ and known ${c_i,d_i,k_i}$ (if $u/k$ is fixed and known).
For some $i$---including all $i$ larger than the index of the largest
prime factor of $[u^2+k^2,u^3+k^3,k]$ once $u/k$ is fixed---we 
have $c_i=d_i=k_i=0$.  For these
\begin{subequations}
\begin{eqnarray}
2b_i &=& 3x_i \\
3b_i &=& 2y_i
\end{eqnarray}
\end{subequations}
The first equation requires $2\mid x_i$ and $3\mid b_i$.
The second equation requires $3\mid y_i$ and $2\mid b_i$.
The combination requires $6\mid b_i$. The absence of the $i$-th prime
allows to multiply $b$ by a sixth (or 12th or 18th\ldots)
power of the $i$-th prime. These factors are of no interest to
the construction of fundamental Bhaskara pairs.

In practice we use the Chinese Remainder Theorem (CRT) for all $i$, whether the $c_i$ or $d_i$
are zero or not \cite{OreAMM59,FraenkelPAMS14}. 
Multiply (\ref{eq.modi1}) by 3 and (\ref{eq.modi2}) by 2,
\begin{equation}
3c_i + 6b_i = 6k_i+9x_i \wedge 2d_i + 6b_i = 6k_i+4y_i
\end{equation}
such that the two factors in front of the $b_i$ are the same, and work
modulo 9 in the first equation and modulo 4 in the second:
\begin{subequations} \label{eq.crtpost}
\begin{eqnarray}
6b_i &\equiv & 6k_i-3c_i \pmod 9; \label{eq.crtpost9}\\
6b_i &\equiv & 6k_i-2d_i \pmod 4\label{eq.crtpost4}.
\end{eqnarray}
\end{subequations}
Because 9 and 4 are relatively prime,
the CRT guarantees that an integer $6a_i$ exists.
Furthermore the result will always be a multiple of 6 (hence $a_i$ an integer), because from \eqref{eq.crtpost9}
the equations read modulo 3 we deduce that $6a_i$ is a multiple of 3, and from
\eqref{eq.crtpost4} read modulo 2
that $6a_i$ is a multiple of 2:
\begin{alg}
For each ratio $a/b=u/k$, the prime power decompositions of
$u^2+k^2$ and $u^3+k^3$ generate a unique
exponent $b_i$ of the prime power $p_i^{b_i}$
of 
a conjectured solution $b$.
\end{alg}
We compute $6b_i \pmod {9\times 4}$ by any algorithm \cite{LaiCEE29},
so $b_i$ is determined $\pmod 6$.

The values of $b_i-k_i$ that result from the CRT for the three relevant values of $c_i$ and the two relevant $d_i$
establish Table \ref{tab.crtsol}. The rows and columns
are bi-periodic for both $c_i$ and $d_i$;
the entries depend only on $d_i \pmod 2$ and on $c_i \pmod 3$.
The zero at the top left entry where $d_i$ is a multiple of 2 and $c_i$ a multiple
of 3 means that a prime $p_i$ is ``discarded'' and its associated
sixth power shoved into the $x^3$ and $y^2$ in equation \eqref{eq.ksol}.
That zero in the table purges the non-fundamental solutions.

\begin{table}
\begin{tabular}{l|ll}
$c_i \backslash d_i$ & 0 &  1 \\
\hline
0 & 0 & 3 \\
1 & 4 & 1 \\
2 & 2 & 5 \\
\end{tabular}
\caption{Solutions $b_i-k_i$ to \eqref{eq.crtpost} as a function of $c_i \pmod 3$ and $d_i \pmod 2$.}
\label{tab.crtsol}
\end{table}

\begin{alg}\label{alg.main}
For any fraction $u/k$ of the Farey tree with $(u,k)=1$, construct
the set $\{p_i\}$ of common prime factors of $k$, $u^2+k^2$ and $u^3+k^3$.
Compute the exponents $k_i$, $c_i$ and $d_i$ of their prime power factorizations.
Construct for each $i$ the exponent $b_i$ as the sum of
the entry in Table \ref{tab.crtsol} plus $k_i$,
and compose $b=\prod_i p_i^{b_i}$.
\end{alg}
\begin{remark}\label{rem.exp15}
$u^2+k^2$ and $u^3+k^3$ have no common divisor larger than 2 (see Lemma
\ref{lem.gcd} in the Appendix). So the only case
where $c_i$ and $d_i$ are both nonzero may occur at prime index $i=1$
and if $u$ and $k$ are both odd. For that reason Table \ref{tab.crtsol} never
fathers odd prime powers $p^1$ or $p^5$, and the only odd prime powers
in $b$ of that form are those contributed by the factor $k=\prod_i p_i^{k_i}$.
\end{remark}

\begin{lma}\label{lma.cop23}
Because $k$ has no common prime factors with either $u^2+k^2$
or $u^3+k^3$ according to Lemma \ref{lma.kfac} in the Appendix,
nonzero $k_i$ appear only where $c_i=d_i=0$. 
\end{lma}

This ensures
that in the construction of $b$ all $p_i^{k_i}$ appear as factors
and that $k\mid b$. $a=ub/k$ generated by the algorithm is always an integer.

The step from \eqref{eq.crtpre}---necessary and sufficient for a solution---to
\eqref{eq.crtpost} eliminates $x_i$ and $y_i$ by applying a modular sieve;
the modular sieve reduces \eqref{eq.crtpost} to a necessary condition.
To show that these $b$ are also sufficient and indeed solve the coupled 
Diophantine equations, the
step from \eqref{eq.crtpre} to \eqref{eq.crtpost} must be reversible, such that
all solutions of \eqref{eq.crtpost} also fulfill \eqref{eq.crtpre}. 
Indeed we can find a multiple of 9 and add it to the right hand side
of the equivalence \eqref{eq.crtpost9} such that it becomes an equality,
and we can find a multiple of 4 and add it
to the right hand side of the equivalence \eqref{eq.crtpost4} such that it becomes an equality. Dividing the two equations
by 3 and 2, respectively, turns out to be a constructive proof that the $3x_i$ and $2y_i$ exist, and that they are
multiples of $3$ and $2$:
\begin{thm}\label{thm.uni}
For each given ratio $a/b=u/k$, the Algorithm \ref{alg.main} generates a unique
fundamental 
solution $b$.
\end{thm}

Lemma \ref{lma.cop23} means that the data reduction of \eqref{eq.ksol}
effectively deals only with
\begin{equation}
(u^2+k^2)\frac{b^2}{k^2}=x^3\wedge
(u^3+k^3)\frac{b^3}{k^3}=y^2
\label{eq.bbar}
\end{equation}
with three integers $u^2+k^2$, $u^3+k^3$ and
\begin{equation}
\bar b \equiv b/k .
\end{equation}
Can we generate more solutions by not just copying the prime factors
of $k$ over to $b$ but introducing higher exponents, such
that $b_i-k_i>0$? The prime power decomposition of \eqref{eq.bbar} would demand
that the surplus factor $p_i^{2(b_i-k_i)}$ divides $x^3$
and that the surplus factor $p_i^{3(b_i-k_i)}$ divides $y^2$.
Lemma \ref{lma.cop23} ensures that these are the only contributions
to $x_i^3$ and $y_i^2$, so effectively $b_i-k_i$ must be multiples of 6.
These sixth powers are introduced at the same time to $a=ub/k$;
so that deliberation does not generate any other fundamental pairs.
With a similar reasoning, multiplying $b$ by any prime power of a prime
that is not a prime factor of $k$---but coupled to $c_i \pmod 3$ and 
to $d_i \pmod 2$ via \eqref{eq.crtpre}---admits only further exponents
that are multiples of 6, and again there is no venue for any other fundamental
solutions from that subset of prime factors.
The solutions are indeed unique as claimed by Theorem \ref{thm.uni}.

\subsection{Examples with $u=1$}

The algorithm and results will be illustrated for a set of small $1/k$  and integer
ratios $b/a$
in Tables \ref{tab.k1}--\ref{tab.k6}.
The tables have 4 columns, the prime index $i$, the exponents $c_i$, $d_i$ and $k_i$ defined by the prime factorization
of $u^2+k^2$, of $u^3+k^3$, and of $k_i$,
and the factor $p_i^{b_i}$ generated by the CRT.
``Spectator'' primes, the cases (rows) where $c_i=d_i=k_i=0$, are not tabulated; they would be absorbed in the sixth powers
of non-fundamental solutions.

\subsubsection{u/k=1}

The case $u=k=1$ in Table \ref{tab.k1} reconvenes the
Bashkara Twin Pairs of Theorem \ref{thm.twin}.

\begin{table}
\begin{tabular}{l|llll}
$i$ & $c_i$ & $d_i$ & $k_i$ & $p_i^{b_i}$ \\
\hline
1 & 1 & 1 & 0 & $2^1$ \\
\end{tabular}
\caption{The Chinese remainder solutions for $u/k=1$.
Fundamental solution $b=2$, $a=2$.}
\label{tab.k1}
\end{table}

\subsubsection{u/k=1/2}

Looking at the second line of Table \ref{tab.pfac} we have only 
contributions
for primes $p_2=3$ and $p_3=5$ in Table \ref{tab.k2}.
\begin{table}
\begin{tabular}{l|llll}
$i$ & $c_i$ & $d_i$ & $k_i$ & $p_i^{b_i}$ \\
\hline
1 & 0 & 0 & 1 & $2^1$ \\
2 & 0 & 2 & 0 & $3^0$ \\
3 & 1 & 0 & 0 & $5^4$ \\
\end{tabular}
\caption{The Chinese remainder solutions for $u/k=1/2$. 
Fundamental solution $b=2\times 5^4$, $a=5^4$.}
\label{tab.k2}
\end{table}
From there all solutions of the form $\{a=b/2,b\}$ are given by the set of 
$b=2\times 5^4 s^6$ with
non-negative integers $s$,  where $\{x,y\}=\{5^3s^4,3\times 5^6s^9\}$.

\subsubsection{$u/k=1/3$}
From the line $k=3$ of Table \ref{tab.pfac} we have the contribution from the
prime factors of Table \ref{tab.k3}.
\begin{table}
\begin{tabular}{l|llll}
$i$ & $c_i$ & $d_i$ & $k_i$ & $p_i^{b_i}$ \\
\hline
1 & 1 & 2 & 0 & $2^4$ \\
2 & 0 & 0 & 1 & $3^1$ \\
3 & 1 & 0 & 0 & $5^4$ \\
4 & 0 & 1 & 0& $7^3$ \\
\end{tabular}
\caption{The Chinese remainder solutions for $u/k=1/3$. 
Fundamental solution $b=2^4\times 3\times 5^4\times 7^3$, $a=2^4\times 5^4\times 7^3$.}
\label{tab.k3}
\end{table}

\subsubsection{$k\ge 4$}
The primes of the line $k=4$ of Table \ref{tab.pfac} generate
Table \ref{tab.k4}.
\begin{table}
\begin{tabular}{l|llll}
$i$ & $c_i$ & $d_i$ & $k_i$ & $p_i^{b_i}$ \\
\hline
1 & 0 & 0 & 2& $2^2$ \\
3 & 0 & 1 & 0 &$5^3$ \\
6 & 0 & 1 & 0 &$13^3$ \\
7 & 1 & 0 & 0 &$17^4$ \\
\end{tabular}
\caption{The Chinese remainder solutions for $u/k=1/4$. 
Fundamental solution $b=2^2\times 5^3\times 13^3\times 17^4$, $a=5^3\times 13^3\times 17^4$.}
\label{tab.k4}
\end{table}

Further solutions $(a=b/k,b)$ with $u/k=1/5\ldots 1/6$ are gathered in Tables \ref{tab.k5}--\ref{tab.k6}.
\begin{table}
\begin{tabular}{l|llll}
$i$ & $c_i$ & $d_i$ & $k_i$ & $p_i^{b_i}$ \\
\hline
1 & 1 & 1 & 0 & $2^1$ \\
2 & 0 & 2 & 0 & $3^0$ \\
3 & 0 & 0 & 1 & $5^1$ \\
4 & 0 & 1 & 0 & $7^3$ \\
6 & 1 & 0 & 0 & $13^4$ \\
\end{tabular}
\caption{The Chinese remainder solutions for $u/k=1/5$. Fundamental solution 
$b=2\times 5\times 7^3\times 13^4$, $a=2\times 7^3\times 13^4$.}
\label{tab.k5}
\end{table}

\begin{table}
\begin{tabular}{l|llll}
$i$ & $c_i$ & $d_i$ & $k_i$ & $p_i^{b_i}$ \\
\hline
1 & 0 & 0 & 1 & $2^1$ \\
2 & 0 & 0 & 1 & $3^1$ \\
4 & 0 & 1 & 0 & $7^3$ \\
11 & 0 & 1 & 0 & $31^3$ \\
12 & 1 & 0 & 0 & $37^4$ \\
\end{tabular}
\caption{The Chinese remainder solutions for $u/k=1/6$. Fundamental solution 
$b=2\times 3\times 7^3\times 31^3\times 37^4$, $a=7^3\times 31^3\times 37^4$.}
\label{tab.k6}
\end{table}

\subsection{Examples with $u>1$}
Some cases where the numerator of $u/k$ is $u>1$ and therefore $b$ not
an integer multiple of $a$ are illustrated in Tables \ref{tab.k3.4}--\ref{tab.k64.83}.
\begin{table}
\begin{tabular}{l|llll}
$i$ & $c_i$ & $d_i$ & $k_i$ & $p_i^{b_i}$ \\
\hline
1 & 0 & 0 & 2 & $2^2$ \\
3 & 2 & 0 & 0 & $5^2$ \\
4 & 0 & 1 & 0 & $7^3$ \\
6 & 0 & 1 & 0 & $13^3$ \\
\end{tabular}
\caption{The Chinese remainder solutions for $u/k=3/4$.
 Fundamental solution $b=2^2\times 5^2\times 7^3\times 13^3$, $a=3\times 5^2\times 7^3\times 13^3$.}
\label{tab.k3.4}
\end{table}

\begin{table}
\begin{tabular}{l|llll}
$i$ & $c_i$ & $d_i$ & $k_i$ & $p_i^{b_i}$ \\
\hline
1 & 0 & 0 & 1 & $2^1$ \\
2 & 0 & 0 & 1 & $3^1$ \\
5 & 0 & 1 & 0 & $11^3$ \\
11 & 0 & 1 & 0 & $31^3$ \\
18 & 1 & 0 & 0 & $61^4$ \\
\end{tabular}
\caption{The Chinese remainder solutions for $u/k=5/6$.
 Fundamental solution $b=2\times 3\times 11^3\times 31^3\times 61^4$,
$a=5\times 11^3\times 31^3\times 61^4$.}
\label{tab.k5.6}
\end{table}

\begin{table}
\begin{tabular}{l|llll}
$i$ & $c_i$ & $d_i$ & $k_i$ & $p_i^{b_i}$ \\
\hline
3 & 3 & 0 & 0 & $5^0$ \\
5 & 0 & 0 & 1 & $11^1$ \\
6 & 0 & 1 & 0 & $13^3$ \\
27 & 0 & 1 & 0 & $103^3$ \\
\end{tabular}
\caption{The Chinese remainder solutions for $u/k=2/11$,
$u^2+k^2=5^3$, $u^3+k^3=13\times 103$.
Fundamental solution $b=11\times 13^3\times 103^3$,
$a=2\times 13^3\times 103^3$.}
\label{tab.k2.11}
\end{table}

\begin{table}
\begin{tabular}{l|llll}
$i$ & $c_i$ & $d_i$ & $k_i$ & $p_i^{b_i}$ \\
\hline
1 & 1 & 3 & 0 & $2^1$ \\
2 & 0 & 2 & 0 & $3^0$ \\
6 & 2 & 0 & 0 & $13^2$ \\
7 & 0 & 0 & 1 & $17^1$ \\
21 & 0 & 1 & 0 & $73^3$ \\
\end{tabular}
\caption{The Chinese remainder solutions for $u/k=7/17$.
 Fundamental solution $b=2\times 13^2\times 17\times 73^3$,
$a=2\times 7\times 13^2\times 73^3$.}
\label{tab.k7.17}
\end{table}

\begin{table}
\begin{tabular}{l|llll}
$i$ & $c_i$ & $d_i$ & $k_i$ & $p_i^{b_i}$ \\
\hline
2 & 0 & 2 & 0 & $3^0$ \\
3 & 1 & 0 & 0 & $5^4$ \\
4 & 0 & 2 & 0 & $7^0$ \\
6 & 3 & 0 & 0 & $13^0$ \\
11 & 0 & 1 & 0 & $31^3$ \\
18 & 0 & 1 & 0 & $61^3$ \\
23 & 0 & 0 & 1 & $83^1$ \\
\end{tabular}
\caption{The Chinese remainder solutions for $u/k=2^6/83$,
Example \ref{ex.pair}.}
\label{tab.k64.83}
\end{table}

\section{Table of Fundamental Solutions}
Systematic exploration of ratios $u/k$ 
sorted along increasing $k$ generates Table \ref{tab.kfun}.

The rather larger value of $b$ for $u/k=5/6$ is derived with
Table \ref{tab.k5.6} from the fact that $u^2+k^2$ have a rather
large isolated prime factor ($p_{18}=61$) which enters with its fourth power.

The rather small value of $b$ at $u/k=2/11$ is explained
with Table \ref{tab.k2.11} from the fact that $u^2+k^2$ is a cube,
which does not contribute to $b$ at all because the
exponent is zero for $c_i\equiv 0 \pmod 3$, $d_i\equiv 0 \pmod 2$
in Table \ref{tab.crtsol}.

\begin{table}
\begin{tabular}{rrr}
$a$ & $b$ & $u/k$ \\
\hline
2 & 2 & 1\\ 
625 & 1250 & 1/2\\ 
3430000 & 10290000 & 1/3\\ 
2449105750 & 3673658625 & 2/3\\ 
22936954625 & 91747818500 & 1/4\\ 
56517825 & 75357100 & 3/4\\ 
19592846 & 97964230 & 1/5\\ 
3327950899994 & 8319877249985 & 2/5\\ 
3437223234 & 5728705390 & 3/5\\ 
104677490484 & 130846863105 & 4/5\\ 
19150763710393 & 114904582262358 & 1/6\\ 
2745064044632305 & 3294076853558766 & 5/6\\ 
3975350 & 27827450 & 1/7\\ 
936110884878 & 3276388097073 & 2/7\\ 
26869428369750 & 62695332862750 & 3/7\\ 
4813895358057500 & 8424316876600625 & 4/7\\ 
329402537360 & 461163552304 & 5/7\\ 
54709453541096250 & 63827695797945625 & 6/7\\ 
3305810795625 & 26446486365000 & 1/8\\ 
113394176313 & 302384470168 & 3/8\\ 
689223517385 & 1102757627816 & 5/8\\ 
978549117961625 & 1118341849099000 & 7/8\\ 
274817266734250 & 2473355400608250 & 1/9\\ 
41793444127641250 & 188070498574385625 & 2/9\\ 
176590156053048868 & 397327851119359953 & 4/9\\ 
6143093188763230 & 11057567739773814 & 5/9\\ 
601306443010000 & 773108283870000 & 7/9\\ 
6758920534667005000 & 7603785601500380625 & 8/9\\ 
104372894488263401 & 1043728944882634010 & 1/10\\ 
458710390065569889 & 1529034633551899630 & 3/10\\ 
8357399286061919849 & 11939141837231314070 & 7/10\\ 
49927726291701142521 & 55475251435223491690 & 9/10\\ 
11221334146768 & 123434675614448 & 1/11\\ 
4801442438 & 26407933409 & 2/11\\ 
33528490382546250 & 122937798069336250 & 3/11\\ 
5247317639775500 & 14430123509382625 & 4/11\\ 
1712007269488880 & 3766415992875536 & 5/11\\ 
13496488877215427538 & 24743562941561617153 & 6/11\\ 
587831133723750 & 923734638708750 & 7/11\\ 
58661465201996135000 & 80659514652744685625 & 8/11\\ 
2046772976463486000 & 2501611415677594000 & 9/11\\ 
414446414697850990 & 455891056167636089 & 10/11\\ 
\end{tabular}
\caption{The fundamental solutions for ratios $a/b=u/k$ up to denominator $k=11$.}
\label{tab.kfun}
\end{table}
\clearpage

Multiplications of solutions of Table \ref{tab.kfun}
with common powers $s^6$ and sorting along increasing $b$ leads 
to Table \ref{tab.res}.
Trivial solutions with $a=0$ ($u/k=0$) are not listed.
The fundamental solutions are flagged by $s=1$ and indicate
where Table \ref{tab.kfun} intersects with Table \ref{tab.res}.
\begin{remark}
The list in Table \ref{tab.res} is not proven to be complete
up to its maximum $b$, because only a limited number of ratios $a/b=u/k$ were computed.
\end{remark}

\begin{longtable}{rrrr}
\caption{Bhaskara pairs with $a>0$, $b\le 3\times 10^{10}$ after
scanning the $u/k$ ratios up to denominators $k\le 200'000$. \cite[A106320]{EIS}}\\
$a$ & $b$ & $u/k$ & s\\
\hline 
\endfirsthead
$a$ & $b$ & $u/k$ & s\\
\hline 
\endhead
2 & 2 & 1 & 1\\
128 & 128 & 1 & 2\\
625 & 1250 & 1/2 & 1\\
1458 & 1458 & 1 & 3\\
8192 & 8192 & 1 & 4\\
31250 & 31250 & 1 & 5\\
40000 & 80000 & 1/2 & 2\\
93312 & 93312 & 1 & 6\\
235298 & 235298 & 1 & 7\\
524288 & 524288 & 1 & 8\\
455625 & 911250 & 1/2 & 3\\
1062882 & 1062882 & 1 & 9\\
2000000 & 2000000 & 1 & 10\\
3543122 & 3543122 & 1 & 11\\
2560000 & 5120000 & 1/2 & 4\\
5971968 & 5971968 & 1 & 12\\
9653618 & 9653618 & 1 & 13\\
3430000 & 10290000 & 1/3 & 1\\
15059072 & 15059072 & 1 & 14\\
9765625 & 19531250 & 1/2 & 5\\
22781250 & 22781250 & 1 & 15\\
3975350 & 27827450 & 1/7 & 1\\
33554432 & 33554432 & 1 & 16\\
48275138 & 48275138 & 1 & 17\\
28130104 & 52743945 & 8/15 & 1\\
29160000 & 58320000 & 1/2 & 6\\
68024448 & 68024448 & 1 & 18\\
56517825 & 75357100 & 3/4 & 1\\
94091762 & 94091762 & 1 & 19\\
19592846 & 97964230 & 1/5 & 1\\
128000000 & 128000000 & 1 & 20\\
73530625 & 147061250 & 1/2 & 7\\
171532242 & 171532242 & 1 & 21\\
226759808 & 226759808 & 1 & 22\\
296071778 & 296071778 & 1 & 23\\
163840000 & 327680000 & 1/2 & 8\\
382205952 & 382205952 & 1 & 24\\
488281250 & 488281250 & 1 & 25\\
617831552 & 617831552 & 1 & 26\\
219520000 & 658560000 & 1/3 & 2\\
332150625 & 664301250 & 1/2 & 9\\
774840978 & 774840978 & 1 & 27\\
963780608 & 963780608 & 1 & 28\\
1189646642 & 1189646642 & 1 & 29\\
625000000 & 1250000000 & 1/2 & 10\\
1458000000 & 1458000000 & 1 & 30\\
1775007362 & 1775007362 & 1 & 31\\
254422400 & 1780956800 & 1/7 & 2\\
2147483648 & 2147483648 & 1 & 32\\
1107225625 & 2214451250 & 1/2 & 11\\
920414222 & 2235291682 & 7/17 & 1\\
2582935938 & 2582935938 & 1 & 33\\
3089608832 & 3089608832 & 1 & 34\\
1800326656 & 3375612480 & 8/15 & 2\\
2449105750 & 3673658625 & 2/3 & 1\\
3676531250 & 3676531250 & 1 & 35\\
1866240000 & 3732480000 & 1/2 & 12\\
4353564672 & 4353564672 & 1 & 36\\
3617140800 & 4822854400 & 3/4 & 2\\
5131452818 & 5131452818 & 1 & 37\\
3437223234 & 5728705390 & 3/5 & 1\\
6021872768 & 6021872768 & 1 & 38\\
3016755625 & 6033511250 & 1/2 & 13\\
1253942144 & 6269710720 & 1/5 & 2\\
7037487522 & 7037487522 & 1 & 39\\
2500470000 & 7501410000 & 1/3 & 3\\
8192000000 & 8192000000 & 1 & 40\\
4705960000 & 9411920000 & 1/2 & 14\\
9500208482 & 9500208482 & 1 & 41\\
10978063488 & 10978063488 & 1 & 42\\
9725113750 & 11493316250 & 11/13 & 1\\
12642726098 & 12642726098 & 1 & 43\\
7119140625 & 14238281250 & 1/2 & 15\\
14512627712 & 14512627712 & 1 & 44\\
16607531250 & 16607531250 & 1 & 45\\
18948593792 & 18948593792 & 1 & 46\\
2898030150 & 20286211050 & 1/7 & 3\\
10485760000 & 20971520000 & 1/2 & 16\\
21558430658 & 21558430658 & 1 & 47\\
24461180928 & 24461180928 & 1 & 48\\
4801442438 & 26407933409 & 2/11 & 1\\
27682574402 & 27682574402 & 1 & 49\\
\label{tab.res}
\end{longtable}

\section{Criteria On The Larger Member}
\subsection{Brute Force}
Building a complete table of the $b$ that are solutions up to some maximum
calls for an efficient method to decide whether any candidate $b$ has
an associate $a$ that solves the equations.

The brute force method is rather slow: one could check all individual
$0\le a\le b$ whether the sum $a^2+b^2$ is a cube and whether $a^3+b^3$ 
is a square; this effort grows $\sim b$. A faster brute force method
considers all cubes $x^3$ in the range $b^{2/3}$ up to $(2b)^{2/3}$, derives
the associates $a=\sqrt{x^3-b^2}$ and checks these first whether they
are integer and then whether they solve the equations; 
this effort grows $\sim b^{2/3}$.

\subsection{Removal of Non-fundamental Pairs}
Reverse engineering the results of the previous sections starts from
the the prime power decomposition of $b$. The set of its factors $p_i^{b_i}$
has $\omega(b)$ members, where $\omega(.)$ denotes the number if distinct primes
that divide the argument \cite[A001221]{EIS}. 
For any subset of the $p_i$ where $b_i\ge 6$, we can split off a set of
sixth prime powers that define a factor $s^6$ considered a part of
a non-fundamental solution, and continue to figure out whether $b/s^6$ is
a member of a fundamental pair. For the rest of the section we only
deal with this checking of $b$ as a member of a fundamental pair. Note
that still the prime factor decomposition of $b$ may have prime exponents
that are $\ge 6$.

\subsection{Congruences for Fundamental Pairs}
This set of prime powers of $b$ is divided in an outer decision loop in 
$2^{\omega(b)}$ different ways into two disjoint subsets; one subset defines
the prime powers of $k=\prod_i p_i^{k_i}$, the other the prime powers of
the conjugate $\bar b=p/k$, $\omega(\bar b)=\omega(b)-\omega(k)$.

If the subset of the prime powers of $k$ is chosen to be empty,
$k=u=1$, this reduces to a trivial check whether $b$ is a member of a Bhaskara
Twin Pair of the format of Theorem \ref{thm.twin}.

For each of these candidates $k$ of $b$ we wish to decide
whether an associate coprime $u$ exists that solves \eqref{eq.bbar}.
\begin{itemize}
\item
If the prime power set of $\bar b$ contains
exponents $\equiv \pm 1 \pmod 6$,
we reject the $k$,
because (see Remark \ref{rem.exp15}) it is impossible
to find coprime $u^2+k^2$ and $u^3+k^3$ that complement them to
cubes and squares. (To reject means to book them as not fostering solutions.)
\item
If the prime power set of $\bar b$ contains exponents $\ge 6$
we reject the $k$ because the same prime power appears in $a=u\bar b$
which violates the search criterion for fundamental pairs.
\end{itemize}

\subsubsection{}
The prime power set of $\bar b$ now contains primes with exponent 2, 3 or 4\@.
According to Table \ref{tab.crtsol} the exponent 2 enforces that the prime
factor $p_i^{1+3\times}$ appears in $u^2+k^2 = \prod_i p_i^{c_i}$
to complement $x^3$,
the exponent 4 enforces
that the prime factor $p_i^{2+3\times}$ appears in $u^2+k^2$ to
complement $x^3$, and 
the exponent 3 enforces that the prime factor $p_i^{1+2\times}$ appears
in $u^3+k^3 = \prod_i p_i^{d_i}$ to complement $y^2$.

\begin{itemize}
\item
We reject exponent sets $\{c_i\}$ if they violate 
Lemma \ref{lma.hypot}.
\end{itemize}
This knowledge that some specific primes or prime powers 
appear in the prime power factorization of $u^2+k^2$ or
$u^3+k^3$ is used to narrow down the search set of $u$
because for these known $p_i$ and given $k$ the quadratic and cubic
residues must be
\begin{equation}
u^2 \equiv -k^2 \pmod {p_i}, \, \mathrm{or}\,\mathrm{even} \pmod {p_i^2},
\end{equation}
respectively
\begin{equation}
u^3 \equiv -k^3 \pmod {p_i}.
\end{equation}

\subsubsection{}
The worst case of the analysis occurs if the entire set of prime powers of $b$
is packed into $k$, $k=b$. Then $\bar b=1$ and none of the rejection
criteria above applies. We are facing the original
set of equations
just with the 
additional support information that $k$ is known and that
$u$ and $k$ need to be coprime: 
\begin{equation}
u^2+k^2=x^3 \wedge u^3+k^3=y^2, \quad (u,k)=1
\label{eq.ukcopr}
\end{equation}
\begin{remark}
The solutions $k$ for the first equation are \cite[A282095]{EIS};
the solutions $k$ for the second equation  are 
\cite[A282639]{EIS}.
The task is to find the values that are in both sequences.
\end{remark}
It is unknown whether any solutions to \eqref{eq.ukcopr}---coprime
Bhaskara pairs---exist.

According to Remark \ref{rem.stpar} the parities
of $k$ and $u$ differ, so $u^2+k^2$ is odd.
In any case the prime factors of $x$ are restricted
by Lemma \ref{lma.hypot} and appear with exponents that are multiples of 3;
the prime factor 2 does not appear.
The prime factors of $k$ are known, and
the prime factor set of $u$ is restricted by not intersecting the prime
factor set of $k$. A weak upper limit of the largest prime factor in $u$
is $k$; a weak upper limit of the largest prime factor in $x$ is 
$(2k^2)^{1/3}$.
$u$ and $x$ have no common prime factor (because that
would need to appear also in $k$ and violate co-primality).
Similarly $k$ and $x$ have no common prime factor.

The simplest way to implement a sieve is to work in a loop
over hypothetical prime factors
$p_i|x$ and discard them if $-k^2$ are not quadratic residues 
as required by \eqref{eq.ukcopr}:
\begin{equation}
u^2 \equiv -k^2 \pmod {p_i^3}.
\end{equation}

A support for brute force construction of all solutions to the
first equation in \eqref{eq.ukcopr}---faster than a loop over all coprime $u$---is given by:
\begin{lma}\label{lma.twosquares}
\cite{ChenMC77,Dahmenarxiv1002}
A solution to
\begin{equation}
u^2+k^2 = x^3,\quad (u,k,x)=1,\quad u,k,x\in \mathbb{Z}
\label{eq.ukx3}
\end{equation}
satisfies
\begin{equation}
\{u,k,x\} = \{s(s^2-3t^2), t(3s^2-t^2),t^2+s^2\}
\label{eq.parts}
\end{equation}
for some $s,t\in \mathbb{Z}$ with $(s,t)=1$ and $st\neq 0$.
\end{lma}

\begin{alg}
Loop over all divisors $t$ (of both signs)
of $k$, compute the conjugate divisor $k/t=3s^2-t^2$.
Check that $s$ is integer, else discard $t$.
If $s$ is not coprime to $t$, discard $t$.
Compute $u=s(s^2-3t^2)$ and take the absolute value.
If that absolute value is larger than $k$ or not coprime to $k$, discard $t$,
otherwise a solution of \eqref{eq.ukx3} is found.
\end{alg}

\begin{remark}\label{rem.stpar}
The parities of $s$ and $t$ in \eqref{eq.parts} are different. In detail: If $k$ is 
\begin{itemize}
\item
odd, all divisors $t$ are odd, and the conjugate $3s^2-t^2$ are also odd.
So $3s^2$ are even. 
Therefore $s^2$ must be even and eventually $s$ be even. 
The conjugate $s^2-3t^2$ are odd and $u$ are even.
\item
even, and $t$ is even: Because we
request $s$ to be coprime to $t$,
$s$ must be odd, so $3s^2$ is odd, and the conjugate $3s^2-t^2$ is odd.
The conjugate $s^2-3t^2$ is odd, and $u$ is odd.
\item
even and $t$ is odd, the conjugate $3s^2-t^2$ must be even, so $3s^2$ must be odd
and hence $s$ must be odd. Its conjugate $s^2-3t^2$ is even, so $u$ is even.
This violates $(u,k)=1$ and does not occur.
\end{itemize}
\end{remark}

\section{Summary}
We have shown that for each ratio $a/b$ 
a unique smallest (fundamental) solution of the non-linear coupled diophantine equations \eqref{eq.defn}
exists, which can be constructed 
by modular analysis via the Chinese Remainder Theorem.
We constructed these explicitly for a set of small ratios.

\bibliographystyle{amsplain}
\bibliography{all}

\appendix
\section{Greatest Common Divisors}
\begin{lma}
The greatest common divisor of $1+k$ and $1+k^2$ is
\begin{equation}
(1+k,1+k^2)=
\left\{\begin{array}{ll} 2,& 2\nmid k;\\ 1,& 2\mid k. \end{array} \right.
\end{equation}
\end{lma}
\begin{proof}
The Euclidean Algorithm to construct the greatest common divisor starts
with \cite{GoodmanAMM59}
\begin{equation}
\frac{k^2+1}{k+1} = k-1+\frac{2}{k+1}
\end{equation}
and basically terminates at this step, so $(1+k,1+k^2) = (1+k,2)$.
This is obviously $1$ or $2$ for even and odd $k$ as claimed.
\end{proof}

\begin{lma}\label{lem.gcd}
The greatest common divisor of $u^2+k^2$ and $u^3+k^3$  for coprime $(u,k)=1$ is
\begin{equation}
(u^2+k^2,u^3+k^3)=
\left\{\begin{array}{ll} 2,& 2\mid (k-u);\\ 1,& 2\nmid (k-u). \end{array} \right.
\end{equation}
\end{lma}
\begin{proof}
The first step of the Euclidean Algorithm is
\begin{equation}
\frac{k^3+u^3}{k^2+u^2} = k+\frac{u^2(k-u)}{k^2+u^2},
\end{equation}
so $(u^3+k^3,u^2+k^2)=(u^2(k-u),k^2+u^2)$.
Assume $p^j$ is one of the inquired
common prime power factors of the common divisor such that $p^j\mid (u^2(k-u))$ and
$p^j\mid (k^2+u^2)$, say $k^2+u^2=vp^j$ for some $j> 0$, $v>0$. 
The first requirement induces $p\mid u^2$ or $p|(k-u)$.
\begin{itemize}
\item
Suppose $p\mid u^2$, then $p\mid u$ by the uniqueness of prime factorizations,
say $u = \alpha p$.
Insertion
of this into $k^2+u^2=vp^j$ and evaluating both sides modulo $p$
leads to the requirement $k^2 \equiv 0 \pmod p$, therefore $p\mid k$.
This contradicts the requirement 
$p\mid u$
because $k$ and $u$
are coprime and must not have a common factor $p$. In conclusion
$p\nmid u^2$.
\item
Since $p\nmid u^2$, $p^j | (u^2(k-u))$
requires $p^j\mid (k-u)$. Rewrite
$k^2+u^2=2u^2-2u(k-u)+(k-u)^2 =vp^j$. Working modulo $p^j$ this becomes
$2u^2 \equiv 0 \pmod {p^j}$. Since $p$ does not divide $u^2$ as shown in the
previous bullet, this requirement reduces to $2\equiv 0 \pmod {p^j}$, leaving
$p^j=2^1$ as the only common prime divisor candidate.
\end{itemize}
It is furthermore obvious that for odd $k$ and odd $u$ both $u^2+k^2$ and
$u^3+k^3$ are even, so the common prime factor $2$ is indeed achieved.
\end{proof}

\begin{lma}\label{lma.kfac}
If $(u,k)=1$,
$k$ is coprime to $u^2+k^2$ and coprime to $u^2+k^3$.
\end{lma}
\begin{proof}
In the first case the first step of the Euclidean Algorithm to compute
$(k^2+u^2,k)$ is
\begin{equation}
\frac{k^2+u^2}{k} = k+\frac{u^2}{k},
\end{equation}
in the second case
\begin{equation}
\frac{k^3+u^3}{k} = k^2+\frac{u^3}{k}.
\end{equation}
So the greatest common divisors are $(k^2+u^2,k)=(u^2,k)$
and $(k^3+u^3,k)=(u^3,k)$. Both expressions equal 1 because we assume
that $u$ and $k$ are coprime.
\end{proof}

\section{Sum of two Squares}
\begin{lma}
There are no solutions to $4k+1+p^2=x^3$ with $p$ a prime and $4k+1<p^2$.
\end{lma}
\begin{proof}
This is obvious for the even prime where $k=0$ is the only candidate.
The other primes are either of the form
$p=4m+1$ with $4k+1+p^2=2(1+2k+4m+8m^2)$ or $q=4m+3$ with $4k+1+q^2=2(5+2k+12m+8m^2)$.
In any case $4k+1+p^2$ is two times an odd number for odd primes $p$.  Because $(4k+1,p)=1$,
Lemma \ref{lma.hypot} applies and the 2 must appear on the right hand side either not
at all or risen to the first power. Both contradicts the request for a perfect cube $x^3$
on the right hand side.
\end{proof}

\begin{lma}
There are no solutions to
\begin{equation}
u^2+p^2=x^3
\end{equation}
where $p$ is a prime, $(u,p)=1$ and $1\le u\le p$.
\end{lma}
\begin{proof}
The case of the even prime is obvious because $1+2^2$ is not a cube,
and the case of the only prime with $3\mid p$ is also obvious
because $1+3^2$ and $2+3^2$ are not cubes. The proof is based on
the failure to create any of the parameterizations required by 
Lemma \ref{lma.twosquares} considering all $t|p$ one by one:
\begin{itemize}
\item
$t=1$  leads to the conjugate divisors $p^2/t=p^2=3s^2-1$.
The other primes fall into the categories $p=3m+1$ where $p^2\equiv 1 \pmod 3$
and $q=3m+2$ where $q^2= 1 \pmod 3$. This contradicts $p^2 \equiv -1 \pmod 3$
of the conjugate required above, so there are no solutions induced by $t=1$.
\item
$t=-1$ leads to a conjugate $p^2/t=-p^2$ which is negative and cannot be
equal to the (essentially) positive $3s^2-1$.
\item
$t=p$ leads to the conjugate divisor $p^2/p=p=3s^2-p^2$,
$p+p^2=p(p+1)=3s^2$. For primes of the form $p=3m+1$ we have $p(p+1)\equiv 2 \pmod 3$ 
and for primes of the form $q=3m+2$ we have $q(q+1) \equiv 0 \pmod 3$.
So only the primes $\equiv 2 \pmod 3$ generate $3s^2$ that are multiples of 3.
If $q=3m+2$ then $q(q+1)=3(m+1)(3m+2)$, so we require $s^2=(m+1)(3m+2)$.
Because $m+1$ and $3m+2$ are coprime, their product can only be
a perfect square $s^2$ if $m+1$ and $3m+2$ are individually perfect squares,
say $m+1=\alpha^2$, $3m+2=\beta^2$, $(\alpha,\beta)=1$. 
$\beta^2-3\alpha^2=-1$.
This negative Pell equation with $D=3$ is not solvable \cite{LagariasTAMS260,FouvryAM172};
the parameterization does not generate solutions.
\item
$t=-p$ leads to the conjugate $p^2/p = -p = 3s^2-p^2$.
$p(p-1)=3s^2$ with $p=3m+1$ implies $p(p-1)\equiv 0 \pmod 3$.
$q(q-1)=3s^2$ with $q=3m+2$ implies $q(q-1)\equiv 2 \pmod 3$.
So only primes $p=3m+1$ remain candidates to represent $3s^2$,
and then $p(p-1)=3m(3m+1)=3s^2$
requires 
$s^2=m(3m+1)=mp$.
Because $m$ and $3m+1$ are coprime, this requires that $p\mid s$.
Because $s$ is a divisor of $u$, this violates the requirement that
$(u,p)=1$ and does not foster solutions.
\end{itemize}

\end{proof}

\end{document}